\documentclass[preprint,12pt]{elsarticlenew}

\usepackage{subfig}
\usepackage{graphicx}
\usepackage{caption,color}   
\usepackage{amssymb}
\usepackage{amsthm}
\usepackage{amsmath}
\usepackage{epic}
\usepackage{setspace}
\usepackage{float}
\usepackage{tikz}
\usepackage{multirow}
\usepackage{hyperref}
\usepackage{graphics}

\usepackage{nicematrix}
\usepackage{comment}

\newtheorem{thm}{Theorem}[section]

\newtheorem{lem}[thm]{Lemma}

\numberwithin{equation}{section}
\journal{}

\begin{document}
\begin{spacing}{1.15}
\begin{frontmatter}
\title{\textbf{~ Matching extension and distance spectral radius\\ }}

\author[label1,label2]{Yuke Zhang}\ead{zhang_yk1029@163.com}
\author[label2]{Edwin R. van Dam}\ead{Edwin.vanDam@tilburguniversity.edu}

\address[label1]{School of Mathematics, East China University of Science and Technology, Shanghai, P.R.~China}
\address[label2]{Department of Econometrics and O.R., Tilburg University, Tilburg, Netherlands\\} 

\begin{abstract}
A graph is called $k$-extendable if each $k$-matching can be extended to a perfect matching. We give spectral conditions for the $k$-extendability of graphs and bipartite graphs using Tutte-type and Hall-type structural characterizations. Concretely, we give a sufficient condition in terms of the spectral radius of the distance matrix for the $k$-extendability of a graph and completely characterize the corresponding extremal graphs. 
A similar result is obtained for bipartite graphs.  
\end{abstract}

\begin{keyword} distance spectral radius, matchings, extendability.\\
\emph{AMS classification(2020):} 05C50 15A18.
\end{keyword}
\end{frontmatter}

\section{Introduction}

In this paper, we give conditions for the extendability of matchings in a graph in terms of the spectral radius of the distance matrix. 
We say that a graph is \textit{$k$-extendable} if each matching consisting of $k$ edges can be extended to a perfect matching. 
Historically, matching extension was born out of the canonical decomposition theory for graphs with perfect matchings \cite{zbMATH04002097}. The study of the concept of $k$-extendability gradually evolved from the concept of so-called elementary (that is, 1-extendable) bipartite graphs. 
 Hetyei \cite{hetyei1964rectangular} provided four useful characterizations of elementary bipartite graphs.
Lov{\'a}sz \cite{lovasz1972structure} showed that the class of bipartite elementary graphs plays an important role in the structure of graphs with a perfect matching.
The first results on $k$-extendable graphs (for arbitrary $k$) were obtained by Plummer \cite{plummer1980n}. In 1980, he studied the properties of $k$-extendable graphs and showed that nearly all $k$-extendable graphs $(k\ge 2)$ are $(k -1)$-extendable and $(k+1)$-connected. Motivated by this work, many researchers further looked at the relationship between $k$-extendability and other graph parameters, e.g., degree \cite{ananchuen1997matching,xu2003degree}, connectivity \cite{lou2004connectivity}, genus \cite{dean1992matching,plummer1988matching} and toughness \cite{plummer1988toughness}. We refer the interested reader also to three surveys \cite{plummer1994extending,plummer1996extending,plummer2008recent} and to the list of references therein.

With the development of spectral graph theory, also the relation between matchings and graph eigenvalues was studied, e.g., \cite{feng2007spectral,liu2018spectral,suil2021spectral} for adjacency eigenvalues,  \cite{brouwer2005eigenvalues,gu2022tight} for Laplacian eigenvalues and \cite{liu2010spectral,zhang2021perfect} for distance eigenvalues. Recently, Fan and Lin \cite{fan2022spectral} investigated the $k$-extendability of graphs from an adjacency spectral perspective. In this paper, we study the relationship between $k$-extendability and the distance spectral radius.

Let $G$ be a graph with vertex set
	$V(G)=\{v_{1}, v_{2}, \ldots,$ $v_{n}\}$ and edge set $E(G)$.
	The \emph{distance} between $v_i$ and $v_j$, denoted by $d({v_i,v_j})$, is the length of a shortest path from $v_i$ to $v_j$. The \emph{distance matrix} of $G$,
	denoted by $D(G)$, is a real symmetric matrix whose $(i,\,j)$-entry is $d({v_i,v_j})$. 
The distance matrix of a graph was introduced by Graham and Pollak \cite{graham1971addressing} to study the routing of messages or data between computers, and this motivated much additional work on the distance matrix.
For example, Merris \cite{merris1990distance} provided an estimation of the spectrum of of the distance matrix of a tree. Since then, there has been a lot of research on distance matrices and their spectra; see the three surveys \cite{aouchiche2014distance,hogben2022spectra,lin2021distancesurvey}.

 By the Perron-Frobenius Theorem, the spectral radius of the distance matrix of $G$, which we denote by $\partial(G)$, equals its largest eigenvalue, and is called the \emph{distance spectral radius} of $G$.
As is usual when studying the distance spectrum, we will assume that the graphs under consideration are connected.

Denote by $\vee$ and $\cup$ the \emph{join} and \emph{union} of two graphs, respectively. Furthermore, we denote by $K_{a,b}\diamond K_{c,d}$ the bipartite graph obtained from the union of $K_{a,b}$ and $ K_{c,d}$ by adding all edges between the parts of the sizes $b$ and $c$. A bipartite graph is called \emph{balanced} if both parts of the bipartition have equal size. Clearly, every bipartite graph with a perfect matching (and hence a $k$-extendable bipartite graph) must be balanced.

Zhang and Lin \cite[Thm.~1.1]{zhang2021perfect} showed that the graph of order $2n$ with the smallest distance spectral radius that does not have a perfect matching is $K_{n-1}\vee(n+1)K_1$ for $n \leq 4$ and $K_1\vee (K_{2n-3}\cup2K_1)$ for $n \geq 5$. They \cite[Thm.~1.2]{zhang2021perfect} also showed that the balanced bipartite graph of order $2n$ with the smallest distance spectral radius that does not have a perfect matching is $K_{n-1,n-2}\diamond K_{1,2}$.

Note that existence of a perfect matching can be considered as $0$-extendability. Instead, we will focus on $k$-extendability for $k\ge 1$ and prove that the graph of order $2n$  with smallest distance spectral radius that is not $k$-extendable is $K_{2k}\vee (K_{2n-2k-1}\cup K_1)$. The balanced bipartite graph of order $2n$ with the smallest distance spectral radius that is not $k$-extendable is $K_{n-k,n-1}\diamond K_{k,1}$.

See Figure \ref{extremalfig} for a picture of the extremal graphs that are not $k$-extendable. It is clear that these graphs are quite different from the extremal graphs that do not have a perfect matching (mentioned above; see Zhang and Lin \cite{zhang2021perfect}), and that they do have perfect matchings. 

\begin{figure}[htp]
	\setlength{\unitlength}{1pt}
\begin{center}
	\begin{picture}(371.2,134.1)

	\put(70.3,62.4){\oval(63.1,26.1)}\put(69.6,109.5){\oval(83.4,29.7)}
	\qbezier(45.0,105.1)(47.9,86.6)(50.8,68.2)
	\qbezier(97.9,105.1)(94.3,86.6)(90.6,68.2)
	\put(70.3,27.6){\circle*{4}}
	\qbezier(50.8,56.6)(60.5,42.1)(70.3,27.6)
	\qbezier(91.4,57.3)(80.8,42.4)(70.3,27.6)

	\put(337.0,47.9){\circle*{4}}

	\qbezier(280.6,99.7)(280.6,94.7)(269.7,91.2)\qbezier(269.7,91.2)(258.9,87.7)(243.6,87.7)\qbezier(243.6,87.7)(228.3,87.7)(217.5,91.2)\qbezier(217.5,91.2)(206.6,94.7)(206.6,99.7)\qbezier(206.6,99.7)(206.6,104.6)(217.5,108.1)\qbezier(217.5,108.1)(228.3,111.7)(243.6,111.7)\qbezier(243.6,111.7)(258.9,111.7)(269.7,108.1)\qbezier(269.7,108.1)(280.6,104.6)(280.6,99.7)

	\qbezier(279.9,51.1)(279.9,46.2)(269.0,42.7)\qbezier(269.0,42.7)(258.2,39.2)(242.9,39.2)\qbezier(242.9,39.2)(227.6,39.2)(216.7,42.7)\qbezier(216.7,42.7)(205.9,46.2)(205.9,51.1)\qbezier(205.9,51.1)(205.9,56.1)(216.7,59.6)\qbezier(216.7,59.6)(227.6,63.1)(242.9,63.1)\qbezier(242.9,63.1)(258.2,63.1)(269.0,59.6)\qbezier(269.0,59.6)(279.9,56.1)(279.9,51.1)

	\qbezier(371.2,101.1)(371.2,95.6)(360.2,91.7)\qbezier(360.2,91.7)(349.1,87.7)(333.5,87.7)\qbezier(333.5,87.7)(317.9,87.7)(306.8,91.7)\qbezier(306.8,91.7)(295.8,95.6)(295.8,101.1)\qbezier(295.8,101.1)(295.8,106.7)(306.8,110.6)\qbezier(306.8,110.6)(317.9,114.6)(333.5,114.6)\qbezier(333.5,114.6)(349.1,114.6)(360.2,110.6)\qbezier(360.2,110.6)(371.2,106.7)(371.2,101.1)

	\qbezier(220.4,95.6)(220.4,75.4)(220.4,58.2)

	\qbezier(264.1,95.3)(264.1,75.0)(264.1,56.8)

	\qbezier(222.4,58.2)(265.4,76.9)(310.3,95.5)

	\qbezier(266.1,56.8)(313.2,75.8)(360.3,95.8)

	\qbezier(312.3,95.5)(324.6,74.7)(337.0,47.9)

	\qbezier(362.3,95.8)(349.6,74.3)(337.0,47.9)
	\put(60.1,66.1){\makebox(0,0)[tl]{$K_{2k}$}}
	\put(50.0,115.5){\makebox(0,0)[tl]{$K_{2n-2k-1}$}}
	\put(15.0,2.0){\makebox(0,0)[tl]{$K_{2k}\vee (K_{2n-2k-1}\cup K_1)$}}
	\put(222.2,105.0){\makebox(0,0)[tl]{\small $(n-k)K_1$}}
	\put(220.5,55.4){\makebox(0,0)[tl]{\small $(n-1)K_1$}}
	\put(326.2,105.0){\makebox(0,0)[tl]{\small $kK_1$}}
	\put(250.2,2.0){\makebox(0,0)[tl]{$K_{n-k,n-1}\diamond K_{k,1}$}}
	\put(66.7,84.1){\circle*{2}}
	\put(74.7,84.1){\circle*{2}}
	\put(70.3,84.1){\circle*{2}}
	\put(66.7,41.3){\circle*{2}}
	\put(74.7,41.3){\circle*{2}}
	\put(70.3,41.3){\circle*{2}}
	\put(240.7,76.9){\circle*{2}}
	\put(248.0,76.9){\circle*{2}}
	\put(244.3,76.9){\circle*{2}}
	\put(331.3,74.7){\circle*{2}}
	\put(340.0,74.7){\circle*{2}}
	\put(335.7,74.7){\circle*{2}}
	\put(290.0,79.8){\circle*{2}}
	\put(295.8,75.4){\circle*{2}}
	\put(292.9,77.6){\circle*{2}}
	\end{picture}
\end{center}
	\caption{The extremal graphs that are not $k$-extendable}
\label{extremalfig}
\end{figure}
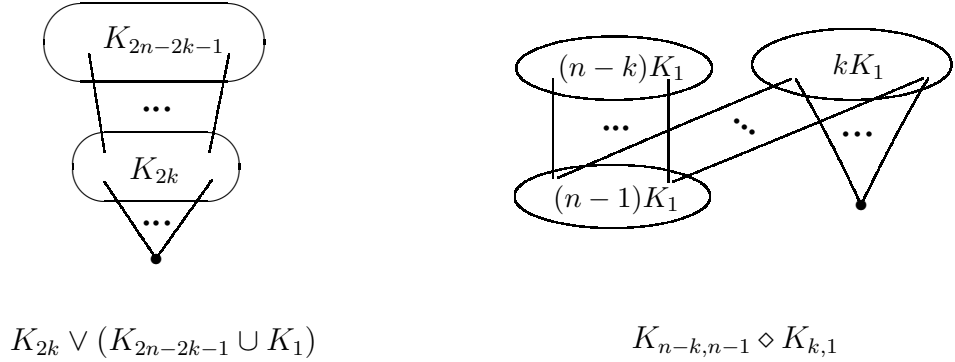

Our paper is further organized as follows: In Section \ref{sec:preliminaries}, we introduce some lemmas on the structure of non-extendable graphs (Section \ref{subsec:structure}) and distance spectral radius (Section \ref{subsec:distance}). In Section \ref{Sec:proof1}, we determine the graph  with the smallest distance spectral radius among all non-extendable graphs of given order (Theorem \ref{thm1}). In Section \ref{Sec:proof2}, we give a sufficient condition in terms of the distance spectral radius for the $k$-extendability of a bipartite graph (Theorem \ref{thm2}). In Section \ref{sec:conclusion}, we finish the paper with an analogous result on $k$-factor-criticality in graphs (Theorem \ref{thm3}).

\section{Preliminaries}\label{sec:preliminaries}

\subsection{The structure of non-extendable graphs}\label{subsec:structure}

We start with a Tutte-type characterization for $k$-extendable graphs obtained by Chen \cite{chen1995binding}. For any
$S \subseteq V (G)$, let $G[S]$ be the subgraph of $G$ induced by $S$ and $G - S$ be the subgraph induced by $V(G)\setminus S$. Denote the
number of odd components in $G$ by  $o(G)$.
\begin{lem}\cite[Lemma 1]{chen1995binding}\label{chenlemma}
Let $k \ge 1$. A graph G is $k$-extendable if and only if
$$o(G - S)\le |S|-2k $$
for any $S \subseteq V(G)$ that contains a $k$-matching.
\end{lem}

A Hall-type condition for bipartite graphs to be $k$-extendable was obtained by Plummer \cite{plummer1986matching}. 
For any $S \subseteq  V (G)$, let $N(S)$ be the set of all neighbors of the vertices in $S$.
\begin{lem}\cite[Thm.~2.2]{plummer1986matching}\label{plummerlemma}
Let $k \geq 1$ and let $G$ be a connected bipartite graph with parts $U$ and $W$. Then the following are equivalent:
\begin{enumerate}[(i)]  \item $G$ is $k$-extendable;\label{s1}
  \item $|U| = |W|$ and for all nonempty subsets $X$ of $U$, if $|X| \le |U| - k$, then
$|N(X)|\ge |X| + k$;\label{s2}
  \item For all $u_1, u_2, \ldots , u_k\in U$ and $w_1, w_2, \ldots , w_k\in W$, the graph $G' = G - \{u_1, \ldots , u_k,w_1, \ldots , w_k\}$ has a perfect matching.
\end{enumerate}
\end{lem}

\subsection{The distance spectral radius}\label{subsec:distance}

An elementary, but fundamental result to compare the distance spectral radii of a graph and a spanning subgraph can be obtained by the Rayleigh quotient and the Perron-Frobenius Theorem.

\begin{lem}\label{PF}
	Let $G$ be a connected graph with $u, v \in V(G)$ and $uv\notin
	E(G)$, then$$\partial(G)>\partial(G+uv).$$
\end{lem}
\begin{proof}
Let $\mathbf{x}$ be a Perron eigenvector for $\partial(G+uv)$, so that $\mathbf{x}$ is positive in all entries. Note that $D(G)=D(G+uv)+M$, where $M$ is a nonzero nonnegative matrix. Then
\begin{align*}
\partial(G)&\ge \frac{\mathbf{x}^\top (D(G+uv)+M) \mathbf{x}}{\mathbf{x}^\top \mathbf{x}}=\frac{\mathbf{x}^\top D(G+uv) \mathbf{x}}{\mathbf{x}^\top \mathbf{x}}+\frac{\mathbf{x}^\top M \mathbf{x}}{\mathbf{x}^\top \mathbf{x}}\\
&> \frac{\mathbf{x}^\top D(G+uv) \mathbf{x}}{\mathbf{x}^\top \mathbf{x}}=\partial(G+uv).
\end{align*}
\end{proof}

We also need a result mentioned as Claim 1 of the proof of Theorem 1.1 in \cite{zhang2021perfect}.
\begin{lem}\cite[pp. 317–319 ]{zhang2021perfect} \label{bh}
	Let $ p\ge 2$ and $ n_i\ge 1 $ for $i=1,\dots,p$.  If $  \sum\limits_{i=1}\limits^{p} n_i =n-s$ where $s\ge 1$, then
	$$ \partial(K_s\vee (K_{n_1} \cup  K_{n_2}\cup \cdots \cup K_{n_p}))\ge \partial(K_s\vee  (K_{n-s-p+1}\cup (p- 1)K_1))$$
	with equality if and only if $n_i=1$ for $i=2,\ldots,p$.
\end{lem}

\section{Extendability and the distance spectral radius of graphs}\label{Sec:proof1}

Using the Tutte-type characterization in Lemma \ref{chenlemma} and the lemmas in Section \ref{subsec:distance} on the distance spectral radius, we will now prove our main result.

\begin{thm}\label{thm1}
 Let $k\ge 1$ and $n\ge k+1$. Let $G$ be a connected graph of order $2n$. If $$\partial(G)\le \partial(K_{2k}\vee (K_{2n-2k-1}\cup K_1)),$$ then $G$ is $k$-extendable unless $G= K_{2k}\vee (K_{2n-2k-1}\cup K_1)$.
\end{thm}

\begin{proof}
Note that $K_{2k}\vee (K_{2n-2k-1}\cup K_1)$ is not $k$-extendable.
We will prove that if $G$ is not $k$-extendable, then $\partial(G)\ge \partial(K_{2k}\vee (K_{2n-2k-1}\cup K_1))$ with equality only if $G = K_{2k}\vee (K_{2n-2k-1}\cup K_1)$.

Suppose that $G$ is not $k$-extendable with $2n$ vertices where $n\ge k+1$, then by Lemma \ref{chenlemma}, there exists some nonempty subset $S$ of $V (G)$, say of size $s$, such that $s\ge  2k$ and $o(G-S) > s-2k$.
Because $G$ has an even order, $o(G-S)$ and $s$ have the same parity, so we have $o(G-S) \ge  s-2k + 2.$
We may assume that all components of $G -S$ are odd, otherwise, we can move one vertex from each even component to the set $S$, and consequently, the
number of odd components and the size of $S$ increase by the same amount, so that all assumption remains valid.
We may also assume that the number of odd components equals $s-2k+2$, for additional odd components (of which there are an even number) may be added to one of the other odd components (as we will not use that a component is connected).  Let the odd number $n_i$ be the cardinality of the $i$-th odd component of $G-S$.
It is clear then that $G$ is a spanning subgraph of $$ K_s\vee (K_{n_1} \cup  K_{n_2}\cup \cdots \cup K_{n_{s-2k+2}})
$$ for some odd integers $ n_1 ,n_2  , \ldots , n_{s-2k+2} $ and $  \sum_{i=1}^{s-2k+2} n_i=2n-s $. By Lemma \ref{PF}, we have
\begin{align*}
\partial(G)\ge \partial(K_s\vee (K_{n_1} \cup  K_{n_2}\cup \cdots \cup K_{n_{s-2k+2}}))
\end{align*}
where equality holds if and only if $G= K_s\vee (K_{n_1} \cup  K_{n_2}\cup \cdots \cup K_{n_{s-2k+2}}).$ 
Write $$G^{(s)}= K_s\vee  (K_{2n-2s+2k-1}\cup (s-2k+1)K_1).$$ By Lemma \ref{bh}, it is clear that
\begin{align*}
\partial(G)\ge \partial(G^{(s)})
\end{align*}
where equality holds if and only if $G=G^{(s)}.$
Note that $2n=s+\sum_{i=1}^{s-2k+2} n_i\ge 2s-2k+2$ and $G^{(2k)}= K_{2k}\vee (K_{2n-2k-1}\cup K_1)$. As the latter is our claimed extremal graph, let us denote its distance spectral radius $\partial(G^{(2k)})$ by $\partial^*$.
The main idea of the following is to show that $\partial(G^{(s)})>\partial^*$ when $n\ge s-k+1$ and $s\ge 2k+1$. 
Let $\mathbf{x}$ be the unit Perron eigenvector of $D(G^{(2k)})$, hence $D(G^{(2k)})\mathbf{x}=\partial^*\mathbf{x}$. It is well-known that $\mathbf{x}$ is constant on each part corresponding to an equitable partition \cite[\S 2.5]{brouwer2011spectra}. 
Thus we may set
\begin{align*}
	\mathbf{x}=( \underbrace{a,\ldots, a}_{2k},\underbrace{b,\ldots, b}_{2n-2k-1},c)^\top
\end{align*}
where $a,b,c\in \mathbb{R}^{+}$.
In order to compare the appropriate spectral radii, we refine the partition, and write $\mathbf{x}$ as follows:
 \begin{align*}
 	\mathbf{x}=(\underbrace{a,\ldots, a}_{2k},\underbrace{b,\ldots, b}_{s-2k},\underbrace{b,\ldots, b}_{2n-2s+2k-1},\underbrace{b,\ldots, b}_{s-2k},c)^\top.
 \end{align*}
Accordingly, $D(G^{(s)})-D(G^{(2k)})$ is partitioned as
{\small \begin{align*}
    \begin{bNiceMatrix}[first-row,first-col]
	& 2k& s-2k&2n-2s+2k-1&s-2k&1\cr
2k&0&0&0&0&0\cr
s-2k&0&0&0&0&-J\cr
2n-2s+2k-1&0&0&0&J&0\cr
s-2k&0&0&J&J-I&0\cr
1&0&-J&0&0&0\cr
    \end{bNiceMatrix}
\end{align*}}
where each $J$ is an all-ones matrix of appropriate size and $I$ is an identity matrix.
Then we have
\begin{align*}
\partial(G^{(s)})-\partial^*&\ge\mathbf{x}^\top(D(G^{(s)})-D(G^{(2k)}))\mathbf{x}\\
&=(s-2k)b\cdot[(4n-3s+2k-3)b-2c].
\end{align*}
Note that $s\ge 2k+1$ and $b>0$, hence it suffices to show that 
\begin{align}\label{ineq:bc}
    (4n-3s+2k-3)b-2c>0.
\end{align}
From the equation $D(G^{(2k)})\mathbf{x}=\partial^*\mathbf{x}$, we obtain that
\begin{align*}
\left\{
\begin{array}{ll}
\partial^* \cdot a=(2k-1)a+(2n-2k-1)b+c\\
\partial^*\cdot b=2ka+(2n-2k-2)b+2c\\
\partial^*\cdot c=2ka+2(2n-2k-1)b\\
\end{array}
\right.
\end{align*}
which implies
$$c=\left(1+\frac{2n-2k-2}{\partial^*+2}\right)b.$$
By substituting this and $s \le n+k-1$ into \eqref{ineq:bc}, it follows that instead of the latter, it suffices to prove that
$$\partial^*> 2+\frac{4}{n-k-2},$$
unless $n=k+2$. But this easily follows from the bound $\partial^*> \min\limits_{i} r_i(D(G^{(2k)}))=2n-1$, where $r_i(A)$ denotes the $i$th row sum of a matrix $A$  (note that $n \geq 4$ if $n \ge k+3$). Note that there is a strict inequality in this bound because the row sums are not constant; we need this strict inequality below.

Thus, except for the case $n=k+2$, which only occurs in conjunction with $n=s-k+1$, the proof is finished.
For the remaining case, the above approach does not work, and the distance spectral radius is quite close to the claimed optimal value. 
If $n=k+2 $ and $n=s-k+1$, then $G^{(s)}=K_{2k+1}\vee 3K_1$, whereas $G^{(2k)}=K_{2k}\vee (K_3\cup K_1)$. Thus, for $G^{(s)}$, there is a clear equitable partition with two parts, and $\partial(G^{(s)})$ is the largest eigenvalue of the corresponding quotient matrix (of $D(G^{(s)})$)
$$\begin{bmatrix}
2k&3\\
2k+1&4
\end{bmatrix}$$
 which has characteristic polynomial 
\begin{align*}
	\phi_s(x)=x^2 - (2k+4)x + 2k - 3.
\end{align*}
Similarly, $\partial^*$ is the largest eigenvalue of the quotient matrix
 $$\begin{bmatrix}
2k-1&3&1\\
2k&2&2\\
2k&6&0
\end{bmatrix}$$
of $D(G^{(2k)})$, and hence $\partial^*$ is the largest root of the characteristic polynomial
\begin{align*}
\phi(x)&=x^3 - (2k + 1)x^2 - (4k+14)x + 4k - 12.
\end{align*}
Next, we let $\varphi_s(x)=(x+3)\phi_s(x).$ In this way, we make sure that $\varphi_s(x)-\phi(x)=-x+2k+3$. As noted before, we have that $\partial^*>2n-1=2k+3$ and hence
\begin{align*}
\phi_s(\partial^*)=\tfrac{1}{\partial^*+3}\varphi_s(\partial^*)=\tfrac{1}{\partial^*+3}(\varphi_s(\partial^*)-\phi(\partial^*))=\tfrac{1}{\partial^*+3}(-\partial^*+2k+3)<0,
\end{align*}
which implies that $\partial(G^{(s)})>\partial^*$.
\end{proof}

\section{{Bipartite graphs}}\label{Sec:proof2}

We will next restrict our attention to bipartite graphs. Instead of the Tutte-type characterization, we will now use the 
Hall-type characterization in Lemma \ref{plummerlemma}.

\begin{thm}\label{thm2}
Let $k\ge 1$ and $n\ge k+1$. Let $G$ be a connected balanced bipartite graph of order $2n$. If $$\partial(G)\le \partial(K_{n-k,n-1}\diamond K_{k,1}),$$ then $G$ is $k$-extendable unless $G = K_{n-k,n-1}\diamond K_{k,1}$.
\end{thm}

\begin{proof}
Note that the bipartite graph $K_{n-k,n-1}\diamond K_{k,1}$ is not $k$-extendable.
We will prove that if $G$ is bipartite but not $k$-extendable, then $\partial(G)\ge \partial(K_{n-k,n-1}\diamond K_{k,1})$ with equality only if $G \cong K_{n-k,n-1}\diamond K_{k,1}$.

Let $G$  be a balanced connected bipartite graph with parts $U$ and $W$ (each of size $n$).
Suppose that $G$ is not $k$-extendable, then by Lemma \ref{plummerlemma}, there exists some nonempty subset $X$, say of size $s$, of $U$ such that $|N(X)|\le s + k-1$ and $s\le n-k$. We now proceed in a similar way as in Section \ref{Sec:proof1}.
Here we have that $G$ is a spanning subgraph of $$B^{(s)}= K_{s,s+k-1}\diamond K_{n-s,n-s-k+1},$$
and, by Lemma \ref{PF}, we have that
\begin{eqnarray*}
\partial(G)\ge \partial(B^{(s)})
\end{eqnarray*}
where equality holds if and only if $G \cong B^{(s)}$.

It is clear that $B^{(1)} \cong B^{(n-k)}=K_{n-k,n-1}\diamond K_{k,1}$, which is our claimed extremal graph, and let us denote its distance spectral radius $\partial(B^{(1)})$ by $\partial^*$. 
Note that more generally, $B^{(s)} \cong B^{(n-k+1-s)}$, so it suffices to show that $\partial(B^{(s)})>\partial^*$ for $ \frac{1}{2}(n-k+1) \le s \le n-k-1$. Note that such $s$ only occur when $n \geq k+3$.

Let $\mathbf{z}$ be the unit Perron eigenvector of $D(B^{(1)})$, hence $D(B^{(1)})\mathbf{z}=\partial^*\mathbf{z}$. As before,  $\mathbf{z}$ is constant on each part corresponding to an equitable partition \cite[\S 2.5]{brouwer2011spectra}, hence we may set
\begin{align*}
	\mathbf{z}=( \underbrace{a_1,\ldots, a_1}_{n-k},\underbrace{a_2,\ldots, a_2}_{k},\underbrace{b_1,\ldots, b_1}_{n-1},b_2)^\top
\end{align*}
where $a_1,a_2,b_1,b_2\in \mathbb{R}^{+}$. 
Again, we refine the partition, and write
 \begin{eqnarray*}
 		\mathbf{z}=( \underbrace{a_1,\ldots, a_1}_{s}, \underbrace{a_1,\ldots, a_1}_{n-k-s},\underbrace{a_2,\ldots, a_2}_{k},\underbrace{b_1,\ldots, b_1}_{s+k-1},\underbrace{b_1,\ldots, b_1}_{n-s-k},b_2)^\top.
 \end{eqnarray*}
Accordingly, we can partition $D(B^{(s)})-D(B^{(1)})$ as
 {\small \begin{align*}
    \begin{bNiceMatrix}[first-row,first-col]
&s&n-k-s&k&s+k-1&n-s-k&1\cr
s&0&0&0&0&2J&0\cr
n-k-s&0&0&0&0&0&-2J\cr
k&0&0&0&0&0&0\cr
s+k-1&0&0&0&0&0&0\cr
n-s-k&2J&0&0&0&0&0\cr
1&0&-2J&0&0&0&0\cr
    \end{bNiceMatrix}
\end{align*}}
and obtain that
\begin{align}\label{beq}
\partial(B^{(s)})-\partial^*&\ge \mathbf{z}^\top(D(B^{(s)})-D(B^{(1)})\mathbf{z} \notag \\ 
&= 4(n-s-k)a_1(sb_1-b_2).
\end{align}
Since $D(B^{(1)})\mathbf{z}=\partial^*\mathbf{z}$, we find that
\begin{align*}
\left\{
  \begin{array}{ll}
\partial^*\cdot b_1=(n-k)a_1+ka_2+2(n-2)b_1+2b_2,\\
\partial^*\cdot b_2=3(n-k)a_1+ka_2+2(n-1)b_1.\\
  \end{array}
\right.
\end{align*}

Because $s \ge \frac{1}{2}(n-k+1)$, we may assume that $s \geq 3$, except for the case that both $n=k+3$ and $s=2$ (which case we will discuss below).
If indeed $s \ge 3$, then
\begin{align*}
\partial^*\cdot (sb_1-b_2)\ge \partial^*\cdot (3b_1-b_2)={2ka_2+(4n-10)b_1+6b_2}>0,
\end{align*}
and hence \eqref{beq} shows that $\partial(B^{(s)})-\partial^*>0.$
Therefore, except for the case $s=2$ and $n=k+3$, the proof is finished.

In the remaining case, we will again consider the quotient matrices and their characteristic polynomials. 
We now have $B^{(2)}= K_{2,k+1}\diamond K_{k+1,2}$, which has an equitable partition with four parts, and $\partial(B^{(2)})$ is the largest eigenvalue of the corresponding quotient matrix 
$${\small\begin{bmatrix}
2&2k+2&k+1&6\\
4&2k&k+1&2\\
2&k+1&2k&4\\
6&k+1&2k+2&2
\end{bmatrix}},$$
which has characteristic polynomial
\begin{align*}
\phi_2(x)=x^4 - (4k+4)x^3 + (3k^2 - 6k - 53)x^2 + (12k^2 + 88k - 68)x- 36k^2 + 72k - 20.
\end{align*}
Similarly, $B^{(1)}= K_{3,k+2}\diamond K_{k,1}$ has quotient matrix (of $D(B^{(1)})$) 
$${\small\begin{bmatrix}
0&2k+4&k&9\\
2&2k+2&k&3\\
1&k+2&2k-2&6\\
3&k+2&2k+2&4
\end{bmatrix}},$$
with characteristic polynomial 
\begin{align*}
	\phi_1(x)=x^4 - (4k+4)x^3 + (3k^2 - 6k - 45)x^2 + (12k^2 + 72k - 52)x - 24k^2 + 64k + 28.
\end{align*}
Note that $\partial^*$ is the largest root of $\phi_1(x)$
and $\partial^*> \min\limits_{i} r_i(D(B^{(1)})=3k+7$.
Then
\begin{align*}
\phi_2(\partial^*)&=\phi_2(\partial^*)-\phi_1(\partial^*)
=-8{\partial^*}^2 + (16k - 16)\partial^* - 12k^2 + 8k - 48\\
&\le -8(3k+7)^2 + (16k - 16)(3k+7) - 12k^2 + 8k - 48\\
&<0
\end{align*}
and hence $\partial(B^{(2)})>\partial^*$. 
\end{proof}


\section{Concluding remarks}\label{sec:conclusion}
We have studied the relationship between extendability of matchings and the distance spectral radius of graphs. Related to extendability is the concept of $k$-factor-criticality, which was introduced by Favaron \cite{favaron1996k} and Yu \cite{qinglin1993characterizations}, independently. Based on results of $k$-extendability, Yu \cite{qinglin1993characterizations} generalized the idea of $k$-extendability to $k\tfrac{1}{2}$-extendability for graphs of odd order. Besides, Favaron \cite{favaron1996k} extended some results on factor-critical and bicritical graphs.

A graph $G$ is said to be \emph{$k$-factor-critical}, if $G-S$ has a perfect matching for every subset $S \subseteq V (G)$ with $|S| = k.$ It is clear that if a graph $G$ is $2k$-factor-critical then it must be $k$-extendable. Note that for bipartite graphs, one needs another definition that includes balancedness.

A Tutte-type characterization of $k$-factor-criticality due to Yu \cite{qinglin1993characterizations} and independently Favaron \cite{favaron1996k}, is as follows.
\begin{lem}\cite[Thm.~2.11]{qinglin1993characterizations}\cite[Thm.~3.5]{favaron1996k}
Let $k \ge 1$. A graph $G$ of order $n$ is $k$-factor-critical if and only if $n \equiv k\pmod 2$ and
$$o(G- S) \le  |S| - k$$
for any subset $S \subseteq V (G)$ with $|S| \ge k$.
\end{lem}

By using this characterization and a similar analysis as for Theorem \ref{thm1}, we can obtain a sufficient condition in terms of the distance spectral radius to determine whether a graph is $k$-factor-critical.
\begin{thm}\label{thm3}
 Let $k\ge 1$ and $n \equiv k\pmod 2$ with $n\ge k+2$. Let $G$ be a connected graph of order $n$. If $$\partial(G)\le \partial(K_{k}\vee (K_{n-k-1}\cup K_1)),$$ then $G$ is $k$-factor-critical unless $G= K_{k}\vee (K_{n-k-1}\cup K_1)$.
\end{thm}


\bibliographystyle{plain}
\bibliography{20230308}
\end{spacing}
\end{document}